\documentclass{amsart}
\usepackage{amsfonts,amssymb,amscd,amsmath,enumerate,verbatim,calc}
\usepackage{graphics,graphicx}
\usepackage[small,bf]{caption}
\usepackage{pgf,tikz}
\usetikzlibrary{calc}
\usepackage{mathrsfs}
\usetikzlibrary{arrows}
\usepackage{amsrefs}
\usepackage{float}
\newtheorem{theorem}{Theorem}[section]
\newtheorem{lemma}[theorem]{Lemma}

\newtheorem{corollary}[theorem]{Corollary}
\newtheorem{remark}[theorem]{Remark}
\newcommand{\T}{\mathrm}
\newcommand{\F}{\mathbb{F}}
\newcommand{\Z}{\mathbb{Z}}

\newcommand{\fm}{\frak{m}}
\newcommand{\fn}{\frak{n}}

\begin{document}
\author[M.R. Khorsandi and S.R. Musawi]
{Mahdi Reza ~Khorsandi$^*$ and Seyed Reza ~Musawi}

\title[Nonorientable genus]
{On the nonorientable genus of the generalized unit and unitary Cayley graphs of a commutative ring}
\subjclass[2010]{13M05, 05C10, 05C25} \keywords{Unit graph, unitary Cayley graph, co-maximal graph, projective graph, nonorientable genus.}
\thanks{$^*$Corresponding author}
\thanks{E-mail addresses: khorsandi@shahroodut.ac.ir and r\_musawi@shahroodut.ac.ir}
\maketitle

\begin{center}
{\it Faculty of Mathematical Sciences, Shahrood University of Technology,\\
P.O. Box 36199-95161, Shahrood, Iran.}
\end{center}
\vspace{0.4cm}
\begin{abstract}
 Let $R$ be a commutative ring and let $U(R)$ be multiplicative group of unit elements  of $R$. In 2012,  Khashyarmanesh et al. defined generalized unit and unitary Cayley graph,
 $\Gamma(R, G, S)$, corresponding to a  multiplicative subgroup $G$ of $U(R)$ and a non-empty subset $S$ of $G$ with $S^{-1}=\{s^{-1} \mid s\in S\}\subseteq S$, as the graph with vertex set $R$ and two distinct vertices $x$ and $y$ are adjacent if and only if there exists $s\in S$ such that $x+sy \in G$. In this paper, we characterize all Artinian rings $R$ whose $\Gamma(R,U(R), S)$ is projective. This leads to determine all Artinian rings  whose unit graphs, unitary Cayley garphs and  co-maximal graphs are projective.
Also,  we prove that for an  Artinian ring $R$ whose  $\Gamma(R, U(R), S)$ has  finite  nonorientable genus, $R$ must be a finite ring. Finally, it is proved that for a given positive integer $k$,  the number of finite rings $R$ whose $\Gamma(R, U(R), S)$  has nonorientable genus $k$ is finite.
\end{abstract}
\vspace{0.5cm}
\section{Introduction}
All rings considered in this paper are  commutative rings with non-zero identity.
We denote the ring of integers module $n$ by $\Z_n$ and the finite field with $q$ elements by $\F_q$.
Let $R$ be a ring.  We use $Z(R)$, $U(R)$ and $J(R)$ to denote the set of zero-divisors of $R$,   the set of units of $R$ and the Jacobson radical of $R$, respectively.

The idea of associating a graph to a commutative ring was introduced  by Beck in \cite{Beck}. The relationship between ring theory and graph theory has received significant attention in the literature.  After introducing the {\it zero-divisor graph} by Beck, the authors assigned  the other graphs to  a commutative ring. Sharma and Bhatwadekar in \cite{Comaximal}, defined the {\it co-maximal graph} on $R$ as the graph whose vertex set is $R$ and two distinct vertices $x$ and $y$ are adjacent if and only if $Rx+Ry=R$. Afterward, in \cite{Unit} (resp., in \cite{Unitary}), the authors defined the {\it unit} (resp., {\it unitary Cayley}) {\it graph} , $G(R)$ (resp., $\mathrm{Cay}(R, U(R))$), with vertex set $R$ and two distinct vertices $x$ and $y$ are adjacent if and only if $x+y \in U(R)$ (resp., $x-y \in U(R)$).
The unit and unitary Cayley graph were generalized in \cite{Generalized+unit} as follows. The {\it generalized unit and unitary Cayley graph},
$\Gamma(R, G, S)$, corresponding to a  multiplicative subgroup $G$ of $U(R)$ and a non-empty subset $S$ of $G$ with $S^{-1}=\{s^{-1} \mid s\in S\}\subseteq S$, is the graph
with vertex set $R$ and two distinct vertices $x$ and $y$ are adjacent if and only if there exists $s\in S$ such that $x+sy \in G$. If we omit the word ``distinct'', the corresponding graph is denoted by $\overline{\Gamma}(R,G,S)$. Note that the graph $\Gamma(R,G,S)$ is a subgraph of the  co-maximal graph. For simplicity of notation, we denote $\Gamma(R,U(R),S)$ (resp., $\overline{\Gamma}(R,U(R),S)$) by $\Gamma(R,S)$ (resp., $\overline{\Gamma}(R,S)$).

The {\it genus}, $\gamma(\Gamma)$, of a finite simple graph $\Gamma$ is the minimum non-negative integer $g$ such that $\Gamma$ can be embedded in the sphere with $g$ handles. The {\it crosscap number} ({\it nonorientable genus}), $\widetilde{\gamma}(\Gamma)$, of a finite simple graph $\Gamma$ is the minimum non-negative integer $k$ such that $\Gamma$ can be embedded in the sphere with $k$ crosscaps. The genus (resp., nonorientable genus)  of an infinite graph $\Gamma$ is the supremum of genus (resp., nonorientable genus) of its finite subgraphs (see \cite {White,  Mohar}). The problem of finding the genus of a graph is NP-complete (see \cite{Thomassen}). However, the genus
of graphs that can be embedded in the projective plane can be computed in polynomial time (see \cite{Genus+projective}).

 A genus $0$ graph is called {\it planar graph} and a nonorientable genus $1$ graph is called a {\it projective graph}. In \cite{Genus+comaximal}, H.-J. Wang characterized  all finite rings  whose co-maximal graphs have genus at most one.  Also, H.-J. Chiang-Hsieh in \cite{Projective+zero} determined all finite rings with projective zero-divisor graphs. Similar results are established for total graphs in \cite{Projective+total, Genus+gtotal, Genus+total, Genus2+total}. Planar unit and unitary Cayley graphs were investigated  in \cite{Unit, Unitary, Planar+unitary+infinite, Planar+unit+infinite}. Also, Khashyarmanesh et al. in \cite{Generalized+unit} characterized all finite rings $R$ in which $\Gamma(R,S)$ is planar. Recently, Asir et al. in \cite{Genus+one+generalized+unit,Genus+two+generalized+unit},  determined all finite rings $R$ whose $\Gamma(R,S)$ has genus at most two.
Moreover, finite rings with higher genus unit and unitary Cayley graphs were investigated in \cite{Higher+genus+unit, Genus+unit} and \cite{Higher+genus+unitary}, respectively. In this paper, we characterize all Artinian rings $R$ whose $\Gamma(R,S)$ is projective. This leads to determine all Artinian rings  whose unit graphs, unitary Cayley graphs and  co-maximal graphs are projective.  Also,  we prove that for  an   Artinian ring $R$  with  $\widetilde{\gamma}(\Gamma(R,S))<\infty$, $R$ must be a finite ring. Finally, it is proved that for a given positive integer $k$, the number of finite rings $R$ such that $\widetilde{\gamma}(\Gamma(R,S))=k$ is finite.

\section{Preliminaries}
For a graph $\Gamma$, $V(\Gamma)$ and $E(\Gamma)$ denote the vertex set and edge set of $\Gamma$, respectively. The {\it degree}  of a vertex $v$ in the graph $\Gamma$, denoted by $\T{deg}(v)$, is the number of edges of $\Gamma$ incident with $v$, with each loop at $v$ counting as two edges.
The {\it minimum degree} of $\Gamma$ is the minimum degree among the vertices of $\Gamma$ and is denoted by $\delta(\Gamma)$. A {\it complete graph}
 $\Gamma$ is a simple graph such that all vertices of $\Gamma$ are adjacent. In addition, $K_n$ denotes a complete graph with $n$ vertices. A graph $\Gamma$ is called {\it bipartite} if $V(\Gamma)$ admits a partition into two classes such that the vertices in the same partition class must not be adjacent. A simple bipartite graph in which every two vertices from different partition classes are adjacent is called a {\it complete bipartite graph}, denoted by $K_{m,n}$, where $m$ and $n$ are size of the partition classes. Two simple graphs $\Gamma$ and $\Delta$ are said to be {\it isomorphic}, and written by $\Gamma \cong \Delta$, if there exists a bijection $\varphi:V(\Gamma)\rightarrow V(\Delta)$ such that  $xy \in E(\Gamma)$ if and only if $\varphi(x)\varphi(y) \in E(\Delta)$ for all $x, y \in V(\Gamma)$. A graph $\Gamma$ is called {\it connected} if any two of its vertices are
linked by a path in $\Gamma$. A maximal connected subgraph of $\Gamma$ is called a component of $\Gamma$.

A {\it subdivision} of a graph $\Gamma$ is a graph that can be obtained from $\Gamma$ by replacing (some or all) edges by paths. Two graphs are said to be {\it homeomorphic} if both can be obtained from the same graph by subdivision. Let $\Gamma_1$ and $\Gamma_2$ be two graphs without multiple edges. Recall that the {\it tensor product} $\Gamma=\Gamma_1 \otimes \Gamma_2$ is a graph with vertex set $V(\Gamma)=V(\Gamma_1)\times V(\Gamma_2)$ and two distinct vertices $(u_1,u_2)$ and $(v_1,v_2)$ of $\Gamma$ are adjacent if and only if $u_1v_1 \in E(\Gamma_1)$ and $u_2v_2 \in E(\Gamma_2)$. We refer the reader to \cite{Murty} and \cite{White} for general references on graph theory.

The following results give us some useful information about  nonorientable genus of a graph.
\begin{lemma}\label{l1}{\rm (\cite[Chapter 11]{White})} The following statements hold:
  \begin{itemize}
  \item[(a)] Let $G$ be a graph. Then $\widetilde{\gamma}(G) \leq 2\gamma(G)+1$.
  \item[(b)] If $H$ is a subgraph of $G$, then $\widetilde{\gamma}(H) \leq \widetilde{\gamma}(G)$.
  \item[(c)]$ \widetilde{\gamma}(K_n)= \left \{
      \begin {array}{cl}
        \lceil \frac{1}{6}(n-3)(n-4)\rceil& \T{\it if}\  n\geq3 \ \T{\it and}\  n\neq 7,\\
        3&\T{\it if}\ n=7.
      \end{array} \right.$\\
      In particular, $\widetilde{\gamma}(K_n)=1$ if $n=5,6$.
  \item[(d)]$\widetilde{\gamma}(K_{m,n})=\lceil \frac{1}{2}(m-2)(n-2)\rceil$ if $m,n\geq2$.\\
  In particular, $\widetilde{\gamma}(K_{3,3})=\widetilde{\gamma}(K_{3,4})=1$ and $\widetilde{\gamma}(K_{4,4})=2$.

  \end{itemize}
\end{lemma}

\begin{lemma}\label{l2}{\rm (\cite[Theorem 1 and Corollary 3]{Stahl})} Let $G$ be a graph with components $G_1,G_2, \cdots, G_n$. If for all $i=1,  \dots, n$, $\widetilde{\gamma}(G_i)>2\gamma(G_i)$, then
$$\widetilde{\gamma}(G)=1-n+\sum_{i=1}^n\widetilde{\gamma}(G_i),$$
otherwise,
$$\widetilde{\gamma}(G)=2n-\sum_{i=1}^n \mu(G_i),$$
where $\mu(G_i)=\mathrm{max}\{2-2\gamma(G_i), 2- \widetilde{\gamma}(G_i)\}$.
\end{lemma}
If we combine  Lemma  \ref{l1}(a) and Lemma \ref{l2}, we can conclude the following corollary:
\begin{corollary}\label{c1} Let $G$ be a graph with components $G_1,G_2, \cdots, G_n$. Then
$$1-n+\sum_{i=1}^n\widetilde{\gamma}(G_i)\leq\widetilde{\gamma}(G) \leq \sum_{i=1}^n\widetilde{\gamma}(G_i)$$
\end{corollary}
\begin{lemma}\label{l3}{\rm (\cite [Corollaries 11.7 and 11.8]{White})} Let $G$ be a  connected graph with $p\geq 3$ vertices and $q$ edges. Then $\widetilde{\gamma}(G)\geq \frac{q}{3}-p+2$. In particular, if $G$ has no triangle, then $\widetilde{\gamma}(G)\geq \frac{q}{2}-p+2$.
\end{lemma}
Now, from Corollary \ref{c1} together Lemma \ref{l3}, we obtain the following corollary:
\begin{corollary}\label{c2} Let $G$ be a graph with $n$ components,  $p\geq 3$ vertices and $q$ edges. Then $\widetilde{\gamma}(G)\geq \frac{q}{3}-p+n+1$. In particular, if $G$ has no triangles, then $\widetilde{\gamma}(G)\geq \frac{q}{2}-p+n+1$.
\end{corollary}
The authors in \cite[Lemma 2.2]{Projective+total}, obtained the following lemma (when the graph $G$ is connected), but they used Euler's formula in their proof which is false in nonorientable case (see \cite[p. 144]{White}). Fortunately,  the result is true and we prove it in general case.
We remak here that the Euler's formula also used in \cite{Projective+zero}, which is false in nonorientable case and so the results in \cite{Projective+zero} must be checked again.
\begin{lemma}\label{l4} Let $G$ be a graph with $n$ components and $p \geq 3$ vertices.  Then $$\delta(G) \leq 6+ \frac{6 \widetilde{\gamma}(G)-6(n+1)}{p}.$$
\end{lemma}
\begin{proof} Since $\sum_{v \in V(G)} \mathrm{deg}(v)=2q$, then $p\delta(G)\leq 2q$. Now, by Corollary \ref{c2}, $2q \leq 6(p+\widetilde{\gamma}(G)-(n+1))$. This completes the proof.
\end{proof}
\section{$\Gamma(R, S)$ with finite nonorientable genus}
In this section, first we prove that for  an   Artinian ring $R$  with  $\widetilde{\gamma}(\Gamma(R,S))=k<\infty$, for some non-negative intger $k$, $R$ must be a finite ring. Then, we prove that for a given positive integer $k$, the number of finite rings $R$ such that $\widetilde{\gamma}(\Gamma(R,S))=k$ is finite.
We begin with some basic general properties of $\Gamma(R, G, S)$.
\begin{lemma}\label{l5}{\rm (\cite[Remark 2.4]{Generalized+unit})}
\begin{itemize}
    \item[(a)] For any vertex $x$ of $\Gamma(R,G,S)$, we have the inequalities
     $$\ \ \ \ \ \ \ \ \ \ \ \ \ \ \ \ \ \ \ \ \ \ |G|-1\leq\T{deg}(x)\leq |G||S|.$$
     Furthermore, for any vertex $x$ of $\overline{\Gamma}(R,G,S)$, $\T{deg}(x)\geq|G|$.
   \item[(b)] Suppose that $R_1$ and $R_2$ are rings and, for each $i$ with $i=1,2$, $G_i$ is a subgroup of $U(R_i)$. Also, assume that $S_i$ is a non-empty subset of $G_i$ with $S_i^{-1}\subseteq S_i$.
            \begin{itemize}
              \item[(i)] Then $\Gamma(R_1\times R_2,G_1\times G_2,S_1\times S_2)\cong \overline{\Gamma}(R_1,G_1,S_1)\otimes \overline{\Gamma}(R_2,G_2,S_2).$
              \item[(ii)] Furthermore, whenever $R_1=R_2$, $G_1\subseteq G_2$ and $S_1\subseteq S_2$, then $\Gamma(R_1,G_1,S_1)$ is a subgraph of $\Gamma(R_2,G_2,S_2)$.
             \end{itemize}
   \end{itemize}
\end{lemma}

\begin{lemma}\label{t0.1}{\rm (\cite[Theorem 2.7]{Generalized+unit})} The graph $\Gamma(R,G,S)$ is a complete graph if and only if the following statements hold.
\begin{itemize}
  \item[(a)] $R$ is a field.
  \item[(b)] $G=U(R)$.
  \item[(c)] $|S|\geq 2$ or $S=\{-1\}$.
  \end{itemize}
\end{lemma}

\begin{remark}\label{r2} \rm (\cite[Remark 3.1]{Generalized+unit}) Suppose that $\{x_i+J(R)\}_{i\in I}$ is a complete set of coset representation of $J(R)$. Note that if $x \in U(R) $ and $j \in J(R)$, then $x+j\in U(R)$. Hence, whenever $x_i$ and $x_j$ are adjacent vertices in $\Gamma(R,S)$, then every element of $x_i+J(R)$ is adjacent to every element of $x_j+J(R)$.
\end{remark}

\begin{lemma}{\rm (\cite[Proposition 3.2 and its proof]{Generalized+unit})}\label{l5.2} Let $\fm$ be a maximal ideal of $R$ such that $|\frac{R}{\fm}|=2$. Then the graph $\Gamma(R,S)$ is bipartite. Furthermore, if $R$ is a local ring, then $\Gamma(R,S)$ is a complete bipartite graph  with parts $\fm$ and $1+\fm$.
\end{lemma}

\begin{theorem}\label{t1} Let $R$ be an Artinian ring such that $\widetilde{\gamma}(\Gamma(R,S))=k<\infty$, for some non-negative intger $k$. Then $R$ is a finite ring.
\end{theorem}
\begin{proof}  First suppose that $|J(R)|=1$. In this case, for some positive integr $n$, $R\cong  F_1 \times \dots \times F_n$, where $F_i$'s ($1\leq i \leq n$ ) are fields. Suppose on the contrary $R$ is infinite. Hence, without loss of generality we can assume that $F_1$ is infinite. Let $s=(s_1, \dots, s_n) \in S$ and $k^\prime =\mathrm{max}\{3, 4k\} $. Since $F_1$ is infinite we can choose distinct elements $x_1, \dots , x_{k^\prime} , y_1, \dots, y_{k^\prime} \in F_1$ such that $-s_1y_1, \dots, -s_1y_{k^\prime} \not \in \{x_1,\dots, x_{k^\prime}\}$. Now, every element of the form $(x_i, 1, \dots, 1)$, $i=1, \dots, k^\prime$,  is adjacent to every element of the form $(y_j, 0, \dots, 0)$, $j=1, \dots, k^\prime$, in $\Gamma(R, S)$. Thus,  $K_{k^\prime, k^\prime}$ is a subgraph of $\Gamma(R, S)$ and so by parts (b) and (d) of Lemma \ref{l1}, $k^\prime \leq \sqrt{2k}+2$ which is a contradiction.

Now, suppose that $|J(R)|>1$. Since $0$ is adjacent to $1$ in $\Gamma(R, S)$, by Remark \ref{r2}, every element of $0+J(R)$ is adjacent to every element of $1+J(R)$. Hence, $K_{|J(R)|,|J(R)|}$ is a subgraph of  $\Gamma(R, S)$ and so by  parts (b) and (d) of Lemma \ref{l1}, $|J(R)|\leq \sqrt{2k}+2$. Now, since $R$ is an Artinian ring,  for some positive integr $n$, we can write $R\cong  R_1 \times \dots \times R_n$, where $R_i$'s ($1\leq i \leq n$ ) are local rings.
Thus, $|J(R)|=|J(R_1)|\times \dots \times |J(R_n)|$ and so for all  $i=1, \dots, n$, $|J(R_i)|< \infty$.  On the other hand, for all $i=1, \dots, n$,
  $Z(R_i)=J(R_i)$ and so by \cite[Theorem 1]{Ganesan}, $R_i$ is a finite ring. Hence, $R$ is a finite ring.
\end{proof}
 The following corollary is an immediate consequence from Lemma \ref{l1}(a) and Theorem \ref{t1}.
\begin{corollary}\label{c2.5}Let $R$ be an Artinian ring such that $\gamma(\Gamma(R,S))<\infty$. Then $R$ is a finite ring.
\end{corollary}


\begin{remark}\label{r1} \rm  Let $\Gamma(R, S)$ be a bipartite graph such that $\Gamma(R, S)=\overline{\Gamma}(R, S)$. Then since $\Gamma(\Z_2,\{1\})=\overline{\Gamma}(\Z_2,\{1\}) \cong K_2$, by Lemma \ref{l5}(b)(i) and \cite[Lemma 8.1]{Unitary}, $\Gamma(\Z_2^\ell  \times R, \{1\} \times \cdots \times \{1\} \times S) \cong 2^\ell \Gamma(R, S)$ for all $\ell \geq 0$.  In particular,  for any graph $\Gamma(T, S^\prime)$, we can conclude that $\Gamma(\Z_2^\ell  \times T, \{1\} \times \cdots \times \{1\} \times S^\prime) \cong 2^{\ell-1} \Gamma(\Z_2 \times T, \{1\}\times S^\prime)$, for all $\ell \geq 1$.
\end{remark}

\begin{theorem}\label{t2} Let $R$ be a finite ring and $\widetilde{\gamma}(\Gamma(R,S))=k>0$. Then either
$$|R| \leq 6k-12 \ \ \ \ \  or \ \ \ \ \ R \cong (\Z_2)^\ell \times T, $$
where $0 \leq \ell \leq \mathrm{log} _2k+1$ and $T$ is a ring with $|T| \leq 16$.
\end{theorem}
\begin{proof} By Lemma \ref{l4}, $\delta(\Gamma(R,S))\leq 6+ \frac{6k-12}{|R|}$. If $|R| > 6k-12$, then $\delta(\Gamma(R,S)) \leq 6$ and so by Lemma \ref{l5}(a), $|U(R)| \leq 7$.  Now, since $R$ is a finite ring, we can write $R \cong (\Z_2)^\ell \times T$, where $\ell \geq 0$ and $T$ is a finite ring. Since $|U(R)| \leq 7$, in view of \cite[Theorem 3.8]{Genus+unit} and its proof, $|T| \leq 16$. It will suffice to prove that if $\ell>0$, then $\ell \leq \mathrm{log} _2k+1$.
Since $S=\{1\} \times \cdots \times \{1\} \times S^\prime$, for some $S^\prime \subseteq T$ and $\ell \geq1$, by Remark \ref{r1}, $\Gamma(\Z_2^\ell  \times T, \{1\} \times \cdots \times \{1\} \times S^\prime) \cong 2^{\ell-1} \Gamma(\Z_2 \times T, \{1\}\times S^\prime)$.
Set $t:= \widetilde{\gamma}(\Gamma(\Z_2 \times T, \{1\}\times S^\prime)$.  If $t=1$, then by Lemma \ref{l2}, $k=2^{\ell-1}$ and so $\ell= \mathrm{log} _2k+1$. Now, suppose that $t>1$. By Corollary \ref{c1}, $k \geq  1-2^{\ell-1}+2^{\ell-1}t$. Hence, $k \geq2^{\ell-1}+1$ and so $\ell \leq \mathrm{log} _2(k-1)+1$. This completes the proof.
\end{proof}

\begin{corollary}\label{c3}
Let $R$ be a finite ring such that $\widetilde{\gamma}(\Gamma(R,S))=k>0$. Then $|R|\leq 32k$. In particular, for any positive integer $k$, the number of finite rings $R$ such that $\widetilde{\gamma}(\Gamma(R,S))=k$ is finite.
\end{corollary}
\begin{proof}  If $|R| > 6k-12$, then by Theorem \ref{t2}, $R \cong (\Z_2)^\ell \times T, $
where $0 \leq \ell \leq \mathrm{log} _2k+1$ and $T$ is a ring with $|T| \leq 16$. In this case, $|R|=2^\ell \times |T| \leq 2^\ell \times 16 \leq 32k$. Thus, $|R| \leq \mathrm{max}\{6k-12, 32k\}=32k$.
\end{proof}
The following Corollary is an immediate consequence from Corollary \ref{c3} and Lemma \ref{l1}(a).
\begin{corollary}\label{c3.5}
For a given positive integer $g$,  the number of finite rings $R$ such that $\gamma(\Gamma(R,S))=g$ is finite.
\end{corollary}
\section{$\Gamma(R, S)$ with nonorientable genus one}
A graph $G$ is {\it irreducible} for a surface $S$ if $G$ does not embed in $S$, but any proper subgraph of $G$ does embed in $S$. Kuratowski's theorem states that any graph which is irreducible for the sphere is homeomorphic to either $K_5$ or $K_{3,3}$. Glover, Huneke, and Wang in \cite{103+graphs} have constructed a list of 103 graphs which are irreducible for projective plane. Afterward, Archdeacon  \cite{Arch} showed that their list is complete. Hence a graph  embeds in the projective plane if and only if it contains no subgraph homeomorphic to one of the graphs in the list of 103 graphs in \cite{103+graphs}.

In this section we characterize  all finite rings $R$ whose $\Gamma(R, S)$ is projective. First, we focus in the case that $R$ is local.
\begin{lemma}\label{l6} Let $R$ be a finite ring such that $\widetilde{\gamma}(\Gamma(R,S))=1$.  Then $|U(R)| \leq 6$ and $|J(R)| \leq 3$.
\end{lemma}
\begin{proof}  By Lemma \ref{l5}(a), $|U(R)|-1 \leq \delta( \Gamma(R,S))$ and by Lemma \ref{l4}, $\delta( \Gamma(R,S)) \leq 6-\frac{6}{|R|}$. Thus, $|U(R)| \leq 6$. Now, it is sufficient to prove that  $|J(R)| \leq 3$. From the proof of Theorem  \ref{t1}, follows that either $|J(R)|=1$ or $|J(R)| \leq \sqrt{2}+2$. This completes the proof.
\end{proof}
\begin{corollary}\label{c4} Let $R$  be a finite local ring such that $\widetilde{\gamma}(\Gamma(R,S))=1$. Then $|R| \leq 9$. In addition, if $R$ is a finite field, then $|R|\leq 7$.
\end{corollary}
\begin{proof} Let $\fm$ be the unique maximal ideal of $R$. By Lemma \ref{l6}, $|U(R)| \leq 6$ and $|\fm| \leq 3$. This implies that $|R|=|U(R)|+|\fm| \leq 6+3=9$. In addition, if $R$ is a field, then  $|R|=|U(R)|+1  \leq 6+1=7$.
\end{proof}

\begin{lemma}\label{l7} Let $R$ be a finite local ring which is not a field.
\begin{itemize}
\item [(a)] If $|R|=8$, then $\widetilde{\gamma}(\Gamma(R,S))=2$.
\item[(b)]  If $|R|=9$, then $\widetilde{\gamma}(\Gamma(R,S))\geq2$.
\end{itemize}
\end{lemma}
\begin{proof}
Let $\fm$ be the unique maximal ideal of $R$.
\begin{itemize}
\item[(a)] Since $R$ is not a field, in view of \cite[p. 687]{Local+rings+I}, $|\fm|=4$. Hence, $|\frac{R}{\fm}|=2$ and by Lemma \ref{l5.2}, $\Gamma(R, S)$ is a complete bipartite graph with parts $\fm$ and $1+\fm$. Thus, $\Gamma(R, S) \cong K_{4,4}$ and so by Lemma \ref{l1}(d), $\widetilde{\gamma}(\Gamma(R,S))=2$.
\item[(b)] Since $|R|$ is odd, by \cite[Corollary 2.3]{Generalized+unit}, $2 \in U(R)$. It follows that $0$ is adjacent to $2$. Now, by Remark \ref{r2}, every element of $\fm$ is adjacent to every element of $1+\fm$ and $2+\fm$.
On the other hand, since $R$ is not a field, $|\fm|=3$. Thus, $K_{3,6}$ is a subgraph of $\Gamma(R, S)$ and so  by parts (b) and (d) of Lemma \ref{l1}, $\widetilde{\gamma}(\Gamma(R,S))\geq2$.
\end{itemize}
\end{proof}
\begin{lemma}\label{t3} {\rm (\cite [Theorem 3.7]{Generalized+unit})} Let $R$ be a finite ring. Then $\Gamma(R,S)$ is planar if and only if one of the following conditions holds.
      \begin{itemize}
          \item[(a)] $R\cong (\Z_2)^\ell \times T$, where $\ell \geqslant 0$ and $T$ is isomorphic to one of the following rings:
          $$\Z_2, \Z_3, \Z_4\  or \  \frac{\Z_2[x]}{(x^2)}.$$
          \item[(b)] $R\cong\F_4$.
          \item[(c)] $R\cong (\Z_2)^\ell \times \F_4$, where $\ell > 0$ with $S=\{1\}$.
          \item[(d)] $R\cong\Z_5$ with $S=\{1\}$.
          \item[(e)] $R\cong\Z_3 \times \Z_3$ with $S=\{(1,1)\}$, $S=\{(1,-1)\}$ or $S=\{(-1,1)\}$.
      \end{itemize}
\end{lemma}
\begin{theorem}\label{t4}
Let $R$ be a finite local ring. Then $\Gamma(R, S)$ is projective if and only if $R\cong \Z_5$ with $S\neq \{1\}$.
\end{theorem}
\begin{proof}
Suppose that $\widetilde{\gamma}(\Gamma(R,S))=1$. By Corollary \ref{c4} and Lemma \ref{l7}, $|R| \leq 7$. If either $|R| \leq 4$ or $R\cong \Z_5$ with $S=\{1\}$, then by Lemma \ref{t3}, $\Gamma(R, S)$ is planar which is not projective.  On the other hand, since $R$ is a finite local ring, the order of $R$ is a power of a prime number. Thus, either $R \cong \Z_5$ with $S \neq \{1\}$ or $R \cong \Z_7$. If $R \cong \Z_5$ with $S \neq \{1\}$, then by Lemma \ref{t0.1}, $\Gamma(R, S) \cong K_5$ and so  by Lemma \ref{l1}(c), $\widetilde{\gamma}(\Gamma(R,S))=1$.
Now, suppose that $R \cong \Z_7$. If either $|S|\geq 2$ or $S=\{-1\}$, then by Lemma \ref{t0.1}, $\Gamma(R, S) \cong K_7$ and in this case by  Lemma \ref{l1}(c), $\widetilde{\gamma}(\Gamma(R,S))=3$. If $S=\{1\}$, then $\Gamma(\Z_7, \{1\})$, as shown in Figure \ref{f1},  is isomorphic to the graph $A_2$ which is one of the 103 graphs listed in \cite{103+graphs}. Thus, $\Gamma(\Z_7, \{1\})$ is not projective.
\end{proof}

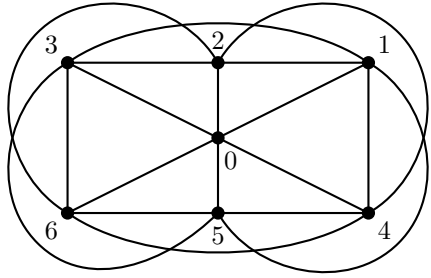
\begin{figure}[ht]
\begin{tikzpicture}[scale=1]
\coordinate (a0) at (0,0);
\fill (a0) circle (2.5pt);
\coordinate (a2) at (0,1);
\fill (a2) circle (2.5pt);
\coordinate (a5) at (0,-1);
\fill (a5) circle (2.5pt);
\coordinate (a1) at (2,1);
\fill (a1) circle (2.5pt);
\coordinate (a3) at (-2,1);
\fill (a3) circle (2.5pt);
\coordinate (a4) at (2,-1);
\fill (a4) circle (2.5pt);
\coordinate (a6) at (-2,-1);
\fill (a6) circle (2.5pt);

\node [above right] at (a1) {1};
\node [above left] at (a3) {3};
\node [below right] at (a4) {4};
\node [below left] at (a6) {6};
\draw[color=black] (.17,-.3) node {0};
\draw[color=black] (0,1.3) node {2};
\draw[color=black] (0,-1.3) node {5};

\draw [line width=.8pt] (a1) .. controls (1, 1.7) and (-1,1.7) .. (a3);
\draw [line width=.8pt] (a4) .. controls (1, -1.7) and (-1,-1.7) .. (a6);
\draw [line width=.8pt] (a2)  to [out=60, in=135] (2.4,1.4) to [out=-45, in=30] (a4);
\draw [line width=.8pt] (a2)  to [out=120, in=45] (-2.4,1.4) to [out=225, in=150] (a6);
\draw [line width=.8pt] (a1)  to [out=-30, in=45] (2.4,-1.4) to [out=225, in=-60] (a5);
\draw [line width=.8pt] (a3)  to [out=210, in=135] (-2.4,-1.4) to [out=-45, in=225] (a5);

\draw [line width=.8pt] (-2,-1) rectangle (2,1);
\draw [line width=.8pt] (-2,-1) -- (2,1);
\draw [line width=.8pt] (-2,1) --(2,-1);
\draw [line width=.8pt] (0,-1) --(0,1);
\end{tikzpicture}
\caption {The graph $\Gamma(\Z_7, \{1\})$.}
\label{f1}
\end{figure}

Now, we determine all finite non-local rings $R$ whose $\Gamma(R, S)$ is projective. First, we state some especial cases.
\begin{lemma}\label{l8} Let $R \cong \Z_2 \times T$ and  $T \in \{\Z_5, \Z_7, \Z_9, \frac{\Z_3[x]}{(x^2)} \}$. Then $\widetilde{\gamma}(\Gamma(R,S))\geq 2$.
\end{lemma}
\begin{proof}
Since $\Gamma(\Z_2, \{1\})=\overline{\Gamma}(\Z_2, \{1\})$, then $\Gamma(R, S)=\overline{\Gamma}(R, S)$
and so by Lemma \ref{l5}(a), for any vertex $x$ of $\Gamma(R, S)$, $\T{deg}(x)\geq |U(R)|$.
On the other hand, since $\fm=\{0\} \times T$ is a maximal ideal of $R$ such that $|\frac{R}{\fm}|=2$, then by Lemma \ref{l5.2}, $\Gamma(R, S)$ is a bipartite graph. Hence, $\Gamma(R, S)$ has no triangles. Now, consider the following cases:

{\bf Case 1}: $T=\Z_5$. In this case $|R|=10$ and $|U(R)|=4$. It follows that $\Gamma(R,S)$ has at least $20$ edges. Thus,  by second part of Corollary \ref{c2}, $\widetilde{\gamma}(\Gamma(R,S))\geq 2$.


{\bf Case 2}: $T=\Z_7$. In this case $|R|=14$ and $|U(R)|= 6$.  Hence, $\Gamma(R,S)$ has at least $42$ edges and so by second part of Corollary \ref{c2}, $\widetilde{\gamma}(\Gamma(R,S))\geq 9$.

 {\bf Case 3}: $T \in \{\Z_9, \frac{\Z_3[x]}{(x^2)}\}$. In this case $|R|=18$ and $|U(R)|= 6$.  Thus, $\Gamma(R,S)$ has at least $54$ edges and so by second part of Corollary \ref{c2}, $\widetilde{\gamma}(\Gamma(R,S))\geq 11$.
\end{proof}

\begin{lemma}\label{l9} Let $R \cong R_1 \times R_2$, $R_1 \in \{\Z_3, \Z_4, \frac{\Z_2[x]}{(x^2)}, \F_4\}$ and $R_2 \in \{\Z_4,\frac{\Z_2[x]}{(x^2)} \}$. Then $\widetilde{\gamma}(\Gamma(R,S))\geq2$.
\end{lemma}
\begin{proof} Since for every $s \in U(R_2)$, $1+s\not \in U(R_2)$, we have $\Gamma(R, S)=\overline{\Gamma}(R, S)$ and so by Lemma \ref{l5}(a), for any vertex $x$ of $\Gamma(R, S)$, $\T{deg}(x)\geq |U(R)|$.
On the other hand, by Lemma \ref{l5.2}, $\Gamma(R, S)$ is a bipartite graph. Indeed, if $\fn$ be  the unique maximal ideal of $R_2$, then $\fm=R_1 \times \fn$ is a maximal ideal of $R$ such that $|\frac{R}{\fm}|=2$. Hence, $\Gamma(R, S)$ has no triangles. Now, consider the following cases:

{\bf Case 1}: $R_1= \Z_3$. In this case $|R|=12$ and $|U(R)|=4$. It follows that $\Gamma(R,S)$ has at least $24$ edges. Thus,  by second part of Corollary \ref{c2}, $\widetilde{\gamma}(\Gamma(R,S))\geq 2$.

 {\bf Case 2}: $R_1 \in \{\Z_4, \frac{\Z_2[x]}{(x^2)}\}$. In this case $|R|=16$ and $|U(R)|=4$. Hence,  $\Gamma(R,S)$ has at least $32$ edges and so by second part of  Corollary \ref{c2}, $\widetilde{\gamma}(\Gamma(R,S))\geq 2$.

 {\bf Case 3}: $R_1= \F_4$. In this case $|R|=16$ and $|U(R)|=6$.  Thus, $\Gamma(R,S)$ has at least $48$ edges and so by second part of Corollary \ref{c2}, $\widetilde{\gamma}(\Gamma(R,S))\geq 10$.
 \end{proof}
 \begin{lemma}\label{l10}
 Let $R\cong \Z_3 \times \F_4$. Then $\widetilde{\gamma}(\Gamma(R,S))\geq2$.
 \end{lemma}
\begin{proof}   Since $S$ is a non-empty subset of $U(R)$ such that $S^{-1}\subseteq S$, there exists  an element $(s_1, s_2)\in S$ where
 $s_1\in \{1, -1\}$ and $(s_1, s_2)^{-1} \in S $. Set $S_2:=\{ s_2, {s_2} ^ {-1}\}$,  $S^\prime:=\{s_1\}\times S_2$ and $G:=\overline{\Gamma}(\Z_3,\{s_1\})$.  Since $\T{Char} (\F_4)=2$, then $|S_2|=1$ if and only if $s_2=-1$ and so by Lemma \ref{t0.1}, $\Gamma(\F_4, S_2) \cong K_4$. On the other hand, by Lemma \ref{l5}(b)(i),
$$\Gamma(R, S^\prime) \cong \overline{\Gamma}(\Z_3,\{s_1\}) \otimes \overline{\Gamma}(\F_4, S_2).$$
Hence,  $G \otimes K_4$ is a subgraph of $\Gamma(R, S^\prime)$.  Note that by Lemma \ref{l5}(b)(i) and Lemma \ref{t0.1},
\begin{align*}
\Gamma(\Z_3 \times \F_4, \{(s_1,-1)\})& = \overline{\Gamma} (\Z_3 \times \F_4, \{(s_1,-1)\})\\
& \cong G\otimes K_4.
\end{align*}
Thus, by Lemma \ref{l5}(a), $G \otimes K_4$ is a 6-regular graph and so by Corollary \ref{c2},
$\widetilde{\gamma}(G \otimes K_4) \geq 2$. Now, by Lemma \ref{l1}(b) and Lemma \ref{l5}(b)(ii), we have the following inequalities:
\begin{align*}
\widetilde{\gamma}(\Gamma(R,S)) & \geq   \widetilde{\gamma}(\Gamma(R,S^\prime))\\
 & \geq  \widetilde{\gamma}(G \otimes K_4)\\
 & \geq 2.
\end{align*}
\end{proof}

\begin{lemma}\label{l11} Let $R \cong \Z_2 \times R_1 \times R_2$ where $R_1$ and $R_2$ are local rings of order $3$ or $4$.
Then $\widetilde{\gamma}(\Gamma(R,S))\geq2$.
\end{lemma}
\begin{proof}  Without loss of generality we can assume that $|R_1| \leq |R_2|$. Since $\Gamma(\Z_2, \{1\})=\overline{\Gamma}(\Z_2, \{1\})$, then $\Gamma(R, S)=\overline{\Gamma}(R, S)$
and so by Lemma \ref{l5}(a), for any vertex $x$ of $\Gamma(R, S)$, $\T{deg}(x)\geq |U(R)|$.
On the other hand, since $\fm=\{0\}\times R_1 \times R_2$ is a maximal ideal of $R$ such that $|\frac{R}{\fm}|=2$, then by Lemma \ref{l5.2}, $\Gamma(R, S)$ is a bipartite graph. Hence, $\Gamma(R, S)$ has no triangles. Now, consider the following cases:

{\bf Case 1}: $|R_1|= |R_2|=3$. In this case $|R|=18$ and $|U(R)|=4$. It follows that $\Gamma(R,S)$ has at least $36$ edges. Thus,  by second part of Corollary \ref{c2}, $\widetilde{\gamma}(\Gamma(R,S))\geq 2$.


{\bf Case 2}: $|R_1|=3$ and $|R_2| =4$. In this case $|R|=24$ and $|U(R)|\geq 4$.  Hence, $\Gamma(R,S)$ has at least $48$ edges and so by second part of Corollary \ref{c2}, $\widetilde{\gamma}(\Gamma(R,S))\geq 2$.

 {\bf Case 3}: $|R_1|= |R_2| =4$. In this case $|R|=32$ and $|U(R)|\geq 4$.  Thus, $\Gamma(R,S)$ has at least $64$ edges and so by second part of Corollary \ref{c2}, $\widetilde{\gamma}(\Gamma(R,S))\geq 2$.
 \end{proof}
 \begin{theorem}\label{t5}
Let $R$ be a non-local finite ring. Then $\Gamma(R, S)$ is projective if and only if $R \cong \Z_3 \times \Z_3$ with $S=\{(-1, -1)\}$, $S=\{(1, -1), (-1, -1)\}$ or $S=\{(-1, 1), (-1, -1)\}$.
\end{theorem}

\begin{proof}
Suppose that $\widetilde{\gamma}(\Gamma(R,S))=1$. Then by Lemma \ref{l6}, $|U(R)| \leq 6$ and $|J(R)| \leq 3$. On the other hand,  by Theorem \ref{t2}, $R \cong (\Z_2)^\ell \times T$,
where $0 \leq \ell \leq 1$ and $T$ is a ring with $|T| \leq 16$. Since $T$ is a finite ring, it is a finite direct product of finite local rings.
Now, by Lemmas \ref{t3}, \ref{l8},  \ref{l9},  \ref{l10} and  \ref{l11}, either
$R \cong \Z_2 \times \F_4$ with $S\neq \{(1, 1)\}$ or $R \cong \Z_3 \times \Z_3$ with $S\neq \{(1,1)\}$, $S\neq\{(1,-1)\}$ and $S \neq \{(-1,1)\}$.

{\bf Case 1}: $R \cong \Z_2 \times \F_4$  with $S\neq \{(1, 1)\}$. In this case $S=\{1\} \times S^\prime$, where $|S^\prime| \geq 2$.
By Lemma \ref{l5}(b)(i),
\begin{align*}
\Gamma(R, S) &\cong \overline{\Gamma}(\Z_2,\{1\}) \otimes \overline{\Gamma}(\F_4, S^\prime)\\
&\cong K_2 \otimes \overline{\Gamma}(\F_4, S^\prime).
\end{align*}
On the other hand, since $\F_4$ is a field and $|S^\prime| \geq 2$, in view of the proof of Lemma \ref{t0.1},
all vertices in $\overline{\Gamma}(\F_4, S^\prime)$ are adjacent with the exception that 0 is not adjacent to 0.  Hence, $\Gamma(R, S)$, as shown in Figure \ref{f2},  is isomorphic to the graph $E_{18}$ which is one of the 103 graphs listed in \cite{103+graphs}. Thus, in this case  $\Gamma(R, S)$ is not projcetive.
\begin{figure}[ht]
\begin{tikzpicture}[scale=1]
\coordinate (a0) at (0,0);
\fill (a0) circle (2.5pt);
\coordinate (a1) at (1,0);
\fill (a1) circle (2.5pt);
\coordinate (a2) at (2,1);
\fill (a2) circle (2.5pt);
\coordinate (a3) at (0,1);
\fill (a3) circle (2.5pt);
\coordinate (a4) at (-2,1);
\fill (a4) circle (2.5pt);
\coordinate (a5) at (-2,-1);
\fill (a5) circle (2.5pt);
\coordinate (a6) at (0,-1);
\fill (a6) circle (2.5pt);
\coordinate (a7) at (2,-1);
\fill (a7) circle (2.5pt);

\draw[color=black] (-.5, -.2) node {$(0, \alpha)$};
\node [right] at (a1) {$(1, 0)$};
\draw[color=black] (2.7,0.9) node {$(1, \alpha^2)$};
\node [above right] at (a3) {$(0, 1)$};
\draw[color=black] (-2.7,0.9) node {$(1, 1)$};
\draw[color=black] (-2.7, -.9) node {$(0, 0)$};
\node [below right] at (a6) {$(1, \alpha)$};
\draw[color=black] (2.7, -.9) node {$(0, \alpha^2)$};

\draw [line width=.8pt] (-2,-1) rectangle (2,1);
\draw [line width=.8pt] (0,0) -- (2,1);
\draw [line width=.8pt] (-2,1) --(0,0);
\draw [line width=.8pt] (0,-1) --(0,0);
\draw [line width=.8pt] (0,0) --(1,0);
\draw [line width=.8pt] (1,0) --(0,1);
\draw [line width=.8pt] (1,0) --(2,-1);
\draw [line width=.8pt] (0,1) -- (0,1.8);
\draw [line width=.8pt] (2,1) --(3,1.5);
\draw [line width=.8pt] (-2,1) --(-3,1.5);
\draw [line width=.8pt] (-2,-1) --(-3,-1.5);
\draw [line width=.8pt] (0,-1) --(0,-1.8);
\draw [line width=.8pt] (2,-1) --(3,-1.5);

\end{tikzpicture}
\caption {The graph $\Gamma(\Z_2 \times \F_4, S)$  with $S\neq \{(1, 1)\}$.}
\label{f2}
\end{figure}
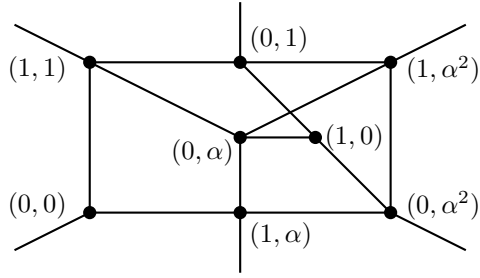

{\bf Case 2}: $R \cong \Z_3 \times \Z_3$ with $S\neq \{(1,1)\}$, $S\neq\{(1,-1)\}$ and $S \neq \{(-1,1)\}$.
In this case by Lemma  \ref{t3}, $\Gamma(R, S)$ is not planar and so $\widetilde{\gamma}(\Gamma(R, S)) \geq 1$. If $S=\{(-1, -1)\}$,  then Figure
 \ref{f3}  shows that  $\Gamma(R, S)$ can be embedded in the projective plane. Thus, $\Gamma(\Z_3 \times \Z_3, \{(-1, -1)\})$ is projective.

\begin{figure}[ht]
\begin{tikzpicture}[scale=1]
\def\r{2.9};
\coordinate (a1) at (0,0);
\fill (a1) circle (2.5pt);
\coordinate (a2) at (1.2,0);
\fill (a2) circle (2.5pt);
\coordinate (a3) at (1.7,-1);
\fill (a3) circle (2.5pt);
\coordinate (a4) at (0,-1);
\fill (a4) circle (2.5pt);
\coordinate (a5) at (-1.5,0);
\fill (a5) circle (2.5pt);
\coordinate (a6) at ({\r*sqrt(2)/2},{\r*sqrt(2)/2});
\fill (a6) circle (2.5pt);
\coordinate (a7) at (\r,0);
\fill (a7) circle (2.5pt);
\coordinate (a8) at ({\r*sqrt(2)/2},-{\r*sqrt(2)/2});
\fill (a8) circle (2.5pt);
\coordinate (a9) at (0,\r);
\fill (a9) circle (2.5pt);
\coordinate (a10) at (-{\r*sqrt(2)/2}, {\r*sqrt(2)/2});\fill (a10) circle (2.5pt);
\coordinate (a11) at (-\r,0);
\fill (a11) circle (2.5pt);
\coordinate (a12) at (-{\r*sqrt(2)/2},- {\r*sqrt(2)/2});
\fill (a12) circle (2.5pt);
\coordinate (a13) at (0,-\r);
\fill (a13) circle (2.5pt);

\draw [line width=.8pt, dashed](0,0) circle (\r cm);
\draw[color=black] (.86, -.3) node {$(0, 0)$};
\node [above] at (a9) {$(2, 2)$};
\node [below] at (a13) {$(2, 2)$};
\draw[color=black] (-.45, .3) node {$(1, 1)$};
\node [right] at (a7) {$(0, 1)$};
\node [above right] at (a6) {$(1, 2)$};
\node [below left] at (a12) {$(1, 2)$};
\node [below right] at (a8) {$(1, 0)$};
\node [above left] at (a10) {$(1, 0)$};
\node [below left] at (a4) {$(0,  2)$};
\node [right] at (a3) {$(2, 1)$};
\node [above] at (a5) {$(2, 0)$};
\node [left] at (a11) {$(0, 1)$};

\draw [line width=.8pt] (a1) -- (a2);
\draw [line width=.8pt] (a3) -- (a2);
\draw [line width=.8pt] (a3) -- (a4);
\draw [line width=.8pt] (a1) -- (a4);
\draw [line width=.8pt] (a1) -- (a9);
\draw [line width=.8pt] (a2) --(a6);
\draw [line width=.8pt] (a2) --(a9);
\draw [line width=.8pt] (a1) --(a5);
\draw [line width=.8pt] (a5) --(a4);
\draw [line width=.8pt] (a5) --(a11);
\draw [line width=.8pt] (a5) -- (a12);
\draw [line width=.8pt] (a4) --(a8);
\draw [line width=.8pt] (a3) --(a8);
\draw [line width=.8pt] (a3) --(a6);
\draw [line width=.8pt] (a11) --(a10);
\draw [line width=.8pt] (a11) --(a12);
\draw [line width=.8pt] (a11) --(a9);
\draw [line width=.8pt] (a10) --(a9);

\end{tikzpicture}
\caption {Embedding  of $\Gamma(\Z_3 \times \Z_3, \{(-1, -1)\})$  in the projective plane.}
\label{f3}
\end{figure}
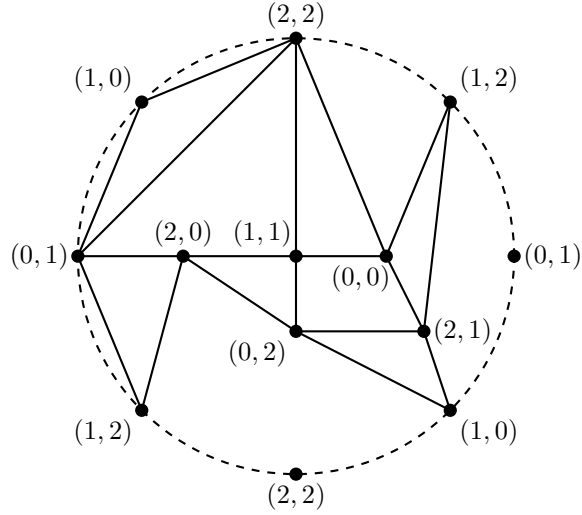
If either $S=\{(1, 1), (-1, -1)\}$ or $S=\{(1, -1), (-1, 1)\}$, then every vertex in $\{(0,0), (0, 1), (0, 2), (1, 0)\}$ is adjacent to every vertex in
$\{(1, 1), (1, 2), (2, 1), (2, 2)\}$.  Hence, $K_{4,4}$ is a subgraph of $\Gamma(R, S)$ and so by parts (b) and (d) of  Lemma \ref{l1},
$\widetilde{\gamma}(\Gamma(R, S)) \geq 2$. If $|S| \geq 3$, then either $\{(1, 1), (-1, -1)\}$ or $\{(1, -1), (-1, 1)\}$ is a subset of $S$
and so by Lemma \ref{l5}(b)(ii) and Lemma \ref{l1}(b), $\widetilde{\gamma}(\Gamma(R, S)) \geq 2$.

If $S=\{(1, 1), (1, -1)\}$, then Figure \ref{f4}  shows that $\Gamma(R, S)$ contains a subgraph isomorphic to $B_3$ which is one of the 103 graphs listed in \cite{103+graphs}. Hence, by Lemma \ref{l1}(b),
$\Gamma(\Z_3 \times \Z_3, \{(1, 1), (1, -1)\})$ is not projective.

\begin{figure}[ht]
\begin{tikzpicture}[scale=1]
\coordinate (a0) at (-1,0);
\fill (a0) circle (2.5pt);
\coordinate (a7) at (1,0);
\fill (a7) circle (2.5pt);
\coordinate (a2) at (0,1);
\fill (a2) circle (2.5pt);
\coordinate (a5) at (0,-1);
\fill (a5) circle (2.5pt);
\coordinate (a1) at (2,1);
\fill (a1) circle (2.5pt);
\coordinate (a3) at (-2,1);
\fill (a3) circle (2.5pt);
\coordinate (a4) at (2,-1);
\fill (a4) circle (2.5pt);
\coordinate (a6) at (-2,-1);
\fill (a6) circle (2.5pt);

\draw[color=black] (-.4,0) node {$(1, 2)$};
\draw[color=black] (1.6,0) node {$(2, 2)$};
\draw[color=black] (-1.9,1.3) node {$(1, 0)$};
\draw[color=black] (-1.9,-1.3) node {$(1, 1)$};
\draw[color=black] (1.9,1.3) node {$(2, 0)$};
\draw[color=black] (1.9,-1.3) node {$(2, 1)$};
\draw[color=black] (0,1.6) node {$(0, 2)$};
\draw[color=black] (-.1,-1.6) node {$(0, 1)$};

\draw [line width=.8pt] (a2)  to [out=60, in=135] (2.4,1.4) to [out=-45, in=30] (a4);
\draw [line width=.8pt] (a2)  to [out=120, in=45] (-2.4,1.4) to [out=225, in=150] (a6);
\draw [line width=.8pt] (a1)  to [out=-30, in=45] (2.4,-1.4) to [out=225, in=-60] (a5);
\draw [line width=.8pt] (a3)  to [out=210, in=135] (-2.4,-1.4) to [out=-45, in=225] (a5);

\draw [line width=.8pt] (-2,-1) rectangle (2,1);
\draw [line width=.8pt] (a6) -- (a0);
\draw [line width=.8pt] (a3) --(a0);
\draw [line width=.8pt] (a0) --(a2);
\draw [line width=.8pt] (a0) --(a5);
\draw [line width=.8pt] (a5) --(a7);
\draw [line width=.8pt] (a7) --(a2);
\draw [line width=.8pt] (a7) --(a1);
\draw [line width=.8pt] (a7) --(a4);

\end{tikzpicture}
\caption {A sugraph of $\Gamma(\Z_3 \times \Z_3, \{(1, 1), (1, -1)\})$.}
\label{f4}
\end{figure}
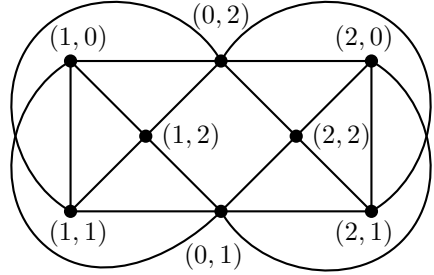

If $S=\{(1, 1), (-1, 1)\}$,
then by  Lemma \ref{l5}(b)(i),
\begin{align*}
\Gamma(R, S)& =\Gamma(\Z_3 \times \Z_3, \{(1, -1\} \times \{ 1\})\\
 &\cong \overline{\Gamma} (\Z_3, \{1, -1\}) \otimes \overline{\Gamma}(\Z_3, \{1\}) \\
 &\cong \overline{\Gamma} (\Z_3, \{1\}) \otimes \overline{\Gamma}(\Z_3, \{1, -1\}) \\
&\cong  \Gamma(\Z_3 \times \Z_3, \{1\} \times \{1, -1\})\\
&= \Gamma(\Z_3 \times \Z_3, \{(1,1), (1, -1)\}).
\end{align*}
This implies that $\widetilde{\gamma}(\Gamma(\Z_3 \times \Z_3, \{(1, 1), (-1, 1)\})) \geq 2$ .

If $S=\{(1, -1), (-1, -1)\}$, then by Figure \ref{f5}, $\Gamma(R, S)$ can be embedded in the projective plane. Hence,
$\Gamma(\Z_3 \times \Z_3, \{(1, -1), (-1, -1)\})$ is projective.

Finally, if  $S=\{(-1, 1), (-1, -1)\}$, then similarly by  Lemma \ref{l5}(b)(i),
$$\Gamma(R, S)  \cong \Gamma(\Z_3 \times \Z_3, \{(1, -1), (-1, -1)\}).$$ It follows that $\Gamma(\Z_3 \times \Z_3, \{(-1, 1), (-1, -1)\})$  is also projective.

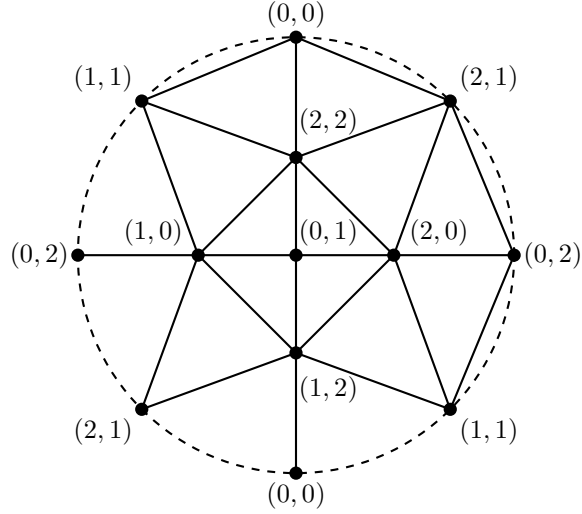
\begin{figure}[ht]
	\begin{tikzpicture}[scale=1]
	\def\r{2.9};
	\def \a{1.3};
	\coordinate (a1) at (0,0);
	\fill (a1) circle (2.5pt);
	\coordinate (a2) at (\a,0);
	\fill (a2) circle (2.5pt);
	\coordinate (a3) at (0,-\a);
	\fill (a3) circle (2.5pt);
	\coordinate (a4) at (0,\a);
	\fill (a4) circle (2.5pt);
	\coordinate (a5) at (-\a,0);
	\fill (a5) circle (2.5pt);
	\coordinate (a6) at ({\r*sqrt(2)/2},{\r*sqrt(2)/2});
	\fill (a6) circle (2.5pt);
	\coordinate (a7) at (\r,0);
	\fill (a7) circle (2.5pt);
	\coordinate (a8) at ({\r*sqrt(2)/2},-{\r*sqrt(2)/2});
	\fill (a8) circle (2.5pt);
	\coordinate (a9) at (0,\r);
	\fill (a9) circle (2.5pt);
	\coordinate (a10) at (-{\r*sqrt(2)/2}, {\r*sqrt(2)/2});\fill (a10) circle (2.5pt);
	\coordinate (a11) at (-\r,0);
	\fill (a11) circle (2.5pt);
	\coordinate (a12) at (-{\r*sqrt(2)/2},- {\r*sqrt(2)/2});
	\fill (a12) circle (2.5pt);
	\coordinate (a13) at (0,-\r);
	\fill (a13) circle (2.5pt);
	
	\draw [line width=.8pt, dashed](0,0) circle (\r cm);
	\draw[color=black] (.43, .3) node {$(0, 1)$};
	\draw[color=black] ({\a+.6}, .3) node {$(2, 0)$};
	\draw[color=black] ({-\a-.6}, .3) node {$(1, 0)$};
	\draw[color=black] (.43,{\a+.48}) node {$(2, 2)$};
	\draw[color=black] (.43,{-\a-.48}) node {$(1, 2)$};
	\node [above] at (a9) {$(0, 0)$};\node [below] at (a13) {$(0, 0)$};
	\node [right] at (a7) {$(0, 2)$};\node [left] at (a11) {$(0, 2)$};
	\node [above right] at (a6) {$(2, 1)$};\node [below left] at (a12) {$(2, 1)$};
	\node [above left] at (a10) {$(1, 1)$};\node [below right] at (a8) {$(1, 1)$};

	\draw [line width=.8pt] (a1) -- (a2);
	\draw [line width=.8pt] (a3) -- (a2);
	\draw [line width=.8pt] (a4) -- (a2);
	\draw [line width=.8pt] (a1) -- (a5);
	\draw [line width=.8pt] (a1) -- (a3);
	\draw [line width=.8pt] (a3) -- (a5);
	\draw [line width=.8pt] (a1) -- (a4);
	\draw [line width=.8pt] (a4) -- (a5);
	\draw [line width=.8pt] (a4) --(a6);
	\draw [line width=.8pt] (a4) --(a10);
	\draw [line width=.8pt] (a5) --(a10);
	\draw [line width=.8pt] (a5) --(a11);
	\draw [line width=.8pt] (a5) --(a12);
	\draw [line width=.8pt] (a3) -- (a12);
	\draw [line width=.8pt] (a4) --(a9);
	\draw [line width=.8pt] (a3) --(a8);
	\draw [line width=.8pt] (a3) --(a13);
	\draw [line width=.8pt] (a2) --(a8);
	\draw [line width=.8pt] (a2) --(a7);
	\draw [line width=.8pt] (a2) --(a6);
	\draw [line width=.8pt] (a9) --(a6);
	\draw [line width=.8pt] (a9) --(a10);
	\draw [line width=.8pt] (a6) --(a7);
	\draw [line width=.8pt] (a7) --(a8);
	
	\end{tikzpicture}
	\caption {Embedding  of $\Gamma(\Z_3 \times \Z_3, \{(1, -1), (-1, -1)\})$  in the projective plane.}
	\label{f5}
\end{figure}

\end{proof}

\begin{corollary} \label{c5}
Let $R$ be a finite ring. Then $\Gamma(R, S)$ is projective if and only if one of the following conditions holds:
\begin{itemize}
          \item[(a)] $R\cong \Z_5$ with $S\neq \{1\}$.
          \item[(b)] $R \cong \Z_3 \times \Z_3$ with $S=\{(-1, -1)\}$, $S=\{(1, -1), (-1, -1)\}$ or $S=\{(-1, 1), (-1, -1)\}$.
\end{itemize}
\end{corollary}
\begin{proof} It is an immediate consequence from Theorems \ref{t4} and \ref{t5}.
\end{proof}

\begin{corollary} There is no finite ring $R$ such that the unit graph $G(R)$ is projective.
\end{corollary}
\begin{proof} It follows from the fact that $G(R)=\Gamma(R, \{1\})$ and Corollary \ref{c5}.
\end{proof}

\begin{corollary} Let  $R$ be a finite ring. Then the unitary Cayley graph $\mathrm{Cay}(R, U(R))$ is projetcive if and only if $R$ is isomorphic to $\Z_5$ or $\Z_3 \times \Z_3$.
\end{corollary}
\begin{proof}
Since $\mathrm{Cay}(R, U(R))=\Gamma(R, \{-1\})$, by Corollary \ref{c5}, there in nothing to prove.
\end{proof}
Some properties of the special case of the graph $\Gamma(R, S)$ in the case that $S=U(R)$ were studied by Naghipour et al. in \cite{Naghipour+unit}.  As a consequence of Corollary \ref{c5}, we have the following corollary.
\begin{corollary}  Let  $R$ be a finite ring. Then $\Gamma(R, U(R))$ is projetcive if and only if $R$ is isomorphic to $\Z_5$.
\end{corollary}


\section{Projective Co-maximal garphs}
Let $R$ be a ring.  As in \cite{Genus+two+generalized+unit}, we denote the co-maximal graph of $R$ by  $C_\Gamma(R)$. For every $S \subseteq U(R)$ with $S^{-1} \subseteq S$, $\Gamma(R, S)$ is a subgraph of $C_\Gamma(R)$.
The authors in \cite{Genus+comaximal} and \cite{Genus+two+generalized+unit} determined all finite rings whose co-maximal graph has genus at most one and genus two, respectively. The question first posed in \cite{Projective+total}  was ``which rings have projective  co-maximal  graphs?".  In this section, we characterize all Artinian rings $R$ whose co-maximal graphs $C_\Gamma(R)$ are  projective.
\begin{theorem}\label{t6}
 Let $R$ be an Artinian ring such that $\widetilde{\gamma}(C_\Gamma(R))<\infty$. Then $R$ is a finite ring.
\end{theorem}
\begin{proof}
Since $\Gamma(R, S)$ is a subgraph of $C_\Gamma(R)$, by Lemma \ref{l1}(b), $\widetilde{\gamma}(\Gamma(R, S)) \leq \widetilde{\gamma}(C_\Gamma(R))$. Hence, $\widetilde{\gamma}(\Gamma(R, S)< \infty$ and so by Theorem \ref{t1}, $R$ is a finite ring.
\end{proof}
The following corollary is an immediate consequence from Lemma \ref{l1}(a) and Theorem \ref{t6}.
\begin{corollary} Let $R$ be an Artinian ring such that $\gamma(C_\Gamma(R))<\infty$. Then $R$ is a finite ring.
\end{corollary}

\begin{remark} { \rm Let $R$ be a finite ring such that $\widetilde{\gamma}(C_\Gamma (R))=k>0$. Since $\Gamma(R, S)$ is a subgraph of
$C_\Gamma(R)$, by Lemma \ref{l1}(b) and Corollary \ref{c3}, $|R|\leq 32k$. In particular, for any positive integer $k$, the number of finite rings $R$ such that $\widetilde{\gamma}(\Gamma(R,S))=k$ is finite.  Also, by Lemma \ref{l1}(a), for a given positive integer $g$,  the number of finite rings $R$ such that $\gamma(C_\Gamma(R))=g$ is finite.
}
\end{remark}

\begin{lemma}\label{l12}{\rm (\cite[Corollary 5.3]{Genus+comaximal})}
Let $R$ be a finite ring. Then $C_\Gamma(R)$ is planar if and only if $R$ is isomorphic to one of the following rings:
$$\Z_2, \ \ \Z_3 , \ \  \Z_4 , \ \  \frac{\Z_2[x]}{(x^2)}, \ \ \F_4, \ \  \Z_2  \times \Z_2 , \ \ \Z_2  \times \Z_3 \ \  or \ \ \Z_2  \times \Z_2 \times \Z_2.$$
\end{lemma}
\begin{theorem}
Let $R$ be a finite ring. Then $C_\Gamma(R)$ is projective if and only if $R$ is isomorphic to one of the following rings:
$$\Z_2 \times \Z_4,  \ \ \ \ \ \Z_2 \times \frac{\Z_2[x]}{(x^2)}  \ \ \ \ \ or \ \ \ \ \ \Z_5.$$
\end{theorem}
\begin{proof}
Suppose that $\widetilde{\gamma}(C_\Gamma(R))=1$. Since $\Gamma(R, S)$ is a subgraph of $C_\Gamma(R)$, by Lemma \ref{l1}(b),
 $\widetilde{\gamma}(\Gamma(R , S)) \leq 1$,  for every $S \subseteq U(R)$ with $S^{-1} \subseteq S$. Thus, by Lemma \ref{t3}, Corollary \ref{c5} and Lemma \ref{l12}, we have the following candidates for $R$:
 \begin{itemize}
          \item[(a)] $\Z_2^\ell$ where $\ell \geq 4$.
          \item[(b)] $(\Z_2)^\ell \times \Z_3$ where $\ell \geq 2$.
          \item[(c)] $(\Z_2)^\ell \times T$ where $\ell \geq 1$ and $T$ is isomorphic to $ \Z_4$ or  $\Z_2[x]/(x^2)$.
           \item[(d)] $\Z_5$.
 \end{itemize}

 {\bf Case 1}: $R \cong \Z_2^\ell$ where $\ell \geq 4$. If $\ell=4$, then Figure \ref{f6} shows that $C_\Gamma(R)$ contains a subdivision of
 $K_{4, 4}$ and so by parts (b) and (d) of Lemma \ref{l1}, $\widetilde{\gamma}(C_\Gamma(R)) \geq 2$. If $\ell \geq 5$, then $C_\Gamma(\Z_2^4)$ is a subgraph of $C_\Gamma(R)$ and so by Lemma \ref{l1}(b), $\widetilde{\gamma}(C_\Gamma(R)) \geq 2$.

 \begin{figure}[ht]
	\begin{tikzpicture}[scale=1]
	\def \a{1.5};
	\def \b{3};
	\def \d{2};
	\coordinate (a1) at (0,0);
	\fill (a1) circle (2.5pt);
	\coordinate (a2) at (0,{\a});
	\fill (a2) circle (2.5pt);
	\coordinate (a3) at (0,{\a*\d});
	\fill (a3) circle (2.5pt);
	\coordinate (a4) at (0,{\a*3});
	\fill (a4) circle (2.5pt);
	\coordinate (b1) at (\b,0);
	\fill (b1) circle (2.5pt);
	\coordinate (b2) at (\b,\a);
	\fill (b2) circle (2.5pt);
	\coordinate (b3) at (\b,{\a*\d});
	\fill (b3) circle (2.5pt);
	\coordinate (b4) at (\b,{\a*3});
	\fill (b4) circle (2.5pt);
	\coordinate (c1) at ({1/2*\b},{\a*3});
	\fill (c1) circle (2.5pt);
	\coordinate (c2) at ({1/4*\b},{\a*\d});
	\fill (c2) circle (2.5pt);
	\coordinate (c3) at ({1/4*\b},{\a*1});
	\fill (c3) circle (2.5pt);
	\node [above] at (c1) {$1100$};\node [below] at (c2) {$1010$}; \node [above] at (c3) {$1001$};
	\node [left] at (a1) {$0111$};\node [right] at (b1) {$1111$};
	\node [left] at (a2) {$0110$};\node [right] at (b2) {$1110$};
	\node [left] at (a3) {$0101$};\node [right] at (b3) {$1101$};
	\node [left] at (a4) {$0011$};\node [right] at (b4) {$1011$};

     \draw [line width=.8pt] (a1) -- (b1);  \draw [line width=.8pt] (a1) -- (b2);
     \draw [line width=.8pt] (a1) -- (b3);  \draw [line width=.8pt] (a1) -- (b4);
      \draw [line width=.8pt] (a2) -- (b1);  \draw [line width=.8pt] (a2) -- (b2);
	  \draw [line width=.8pt] (a2) -- (b3);  \draw [line width=.8pt] (a2) -- (b4);
	  \draw [line width=.8pt] (a3) -- (b1);  \draw [line width=.8pt] (a3) -- (b2);
	  \draw [line width=.8pt] (a3) -- (b3);  \draw [line width=.8pt] (a3) -- (b4);
      \draw [line width=.8pt] (a4) -- (b1);  \draw [line width=.8pt] (a4) -- (b2);
	  \draw [line width=.8pt] (a4) -- (b3);  \draw [line width=.8pt] (a4) -- (b4);
	
	
		\end{tikzpicture}
	\caption {A subgraph of $C_\Gamma(\Z_2 \times \Z_2 \times \Z_2 \times \Z_2)$.}
	\label{f6}
\end{figure}
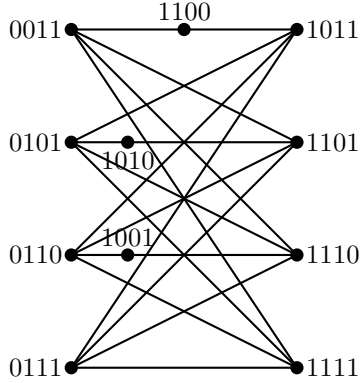

 {\bf Case 2}: $R \cong (\Z_2)^\ell \times \Z_3$ where $\ell \geq 2$. If $\ell=2$, then it is easy to check that the graph $C_\Gamma(R)$
 has $35$ edges (see also \cite[Figure 8]{Genus+two+generalized+unit}) and so by Corollary \ref{c2}, $\widetilde{\gamma}(C_\Gamma(R))\geq 2$. If $\ell \geq 3$, then $C_\Gamma(\Z_2^4)$ is a subgraph of $C_\Gamma(R)$ and so by Lemma \ref{l1}(b),  $\widetilde{\gamma}(C_\Gamma(R)) \geq 2$.

 {\bf Case 3}: $R \cong (\Z_2)^\ell \times T$ where $\ell \geq 1$ and $T$ is isomorphic to $ \Z_4$ or  $\Z_2[x]/(x^2)$. Note that
$C_\Gamma((\Z_2)^\ell \times \Z_4) \cong C_\Gamma((\Z_2)^\ell \times \Z_2[x]/(x^2))$. Hence, it is enough to consider the case
 $R \cong (\Z_2)^\ell \times \Z_4$. If $\ell=1$, then Figure \ref{f7} shows that the graph $C_\Gamma(\Z_2 \times \Z_4)$ can be embedded in the projective plane. If $\ell=2$, then all the vertices $(1,1,0)$, $(1,1,1)$, $(1,1,2)$ and $(1,1,3)$ are adjacent to the vertices $(0,0,1)$, $(0,1,1)$, $(1,0,1)$ and $(1,0,3)$ in $C_\Gamma(R)$. Thus,  $K_{4,4}$ is a subgraph of $C_\Gamma(R)$ and so by parts (b) and (d) of Lemma \ref{l1}, $\widetilde{\gamma}(C_\Gamma(R)) \geq 2$. Finally, if $\ell \geq 3$, then $C_\Gamma(\Z_2 \times \Z_2 \times \Z_4)$ is a subgraph of $C_\Gamma(R)$ and so by Lemma \ref{l1}(b), $\widetilde{\gamma}(C_\Gamma(R)) \geq 2$.

 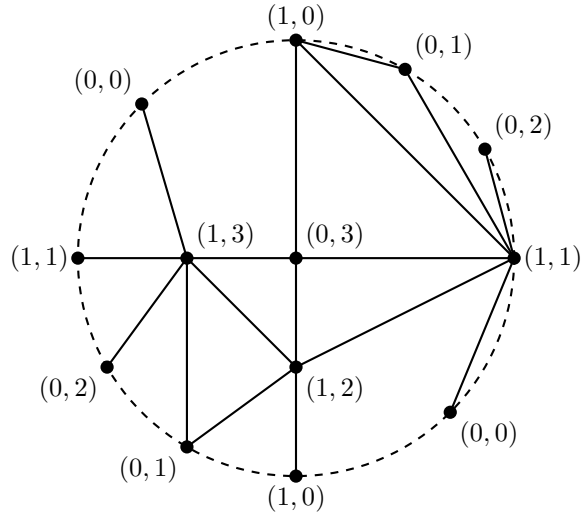
\begin{figure}[ht]
	\begin{tikzpicture}[scale=1]
	\def\r{2.9};
	\def \a{\r/2};
	\coordinate (a1) at (0,0);
	\fill (a1) circle (2.5pt);
	\coordinate (a2) at (-\a,0);
	\fill (a2) circle (2.5pt);
	\coordinate (a3) at (0,-\a);
	\fill (a3) circle (2.5pt);
	\coordinate (a4) at (0,-\r);
	\fill (a4) circle (2.5pt);
	\coordinate (a5) at (0,\r);
	\fill (a5) circle (2.5pt);
	\coordinate (a6) at (\r,0);
	\fill (a6) circle (2.5pt);
	\coordinate (a7) at (-\r,0);
	\fill (a7) circle (2.5pt);
	\coordinate (a8) at (30:\r);
	\fill (a8) circle (2.5pt);
	\coordinate (a9) at (60:\r);
	\fill (a9) circle (2.5pt);
	\coordinate (a10) at (135:\r);
	\fill (a10) circle (2.5pt);
	\coordinate (a13) at ({180+30}:\r);
	\fill (a13) circle (2.5pt);
	\coordinate (a12) at ({180+60}:\r);
	\fill (a12) circle (2.5pt);
	\coordinate (a11) at (-45:\r);
	\fill (a11) circle (2.5pt);

	\draw [line width=.8pt, dashed](0,0) circle (\r cm);
	\draw [line width=.8pt] (a1) -- (a2);
	\draw [line width=.8pt] (a1) -- (a6);
	\draw [line width=.8pt] (a1) -- (a5);
	\draw [line width=.8pt] (a1) -- (a3);
	\draw [line width=.8pt] (a3) -- (a4);
	\draw [line width=.8pt] (a3) -- (a12);
	\draw [line width=.8pt] (a3) --(a6);
	\draw [line width=.8pt] (a3) --(a2);
	\draw [line width=.8pt] (a2) --(a12);
	\draw [line width=.8pt] (a2) --(a10);
	\draw [line width=.8pt] (a2) -- (a13);
	\draw [line width=.8pt] (a2) --(a7);
	\draw [line width=.8pt] (a5) --(a9);
	\draw [line width=.8pt] (a5) --(a6);
	\draw [line width=.8pt] (a6) --(a11);
	\draw [line width=.8pt] (a6) --(a8);
	\draw [line width=.8pt] (a6) --(a9);

	
	\node [above] at (a5) {$(1, 0)$};\node [below] at (a4) {$(1, 0)$};
	\node [right] at (a6) {$(1, 1)$};\node [left] at (a7) {$(1, 1)$};
	\node [below right] at (a11) {$(0, 0)$};\node [above left] at (a10) {$(0, 0)$};
	\node [above right] at (a8) {$(0, 2)$};\node [above right] at (a9) {$(0, 1)$};
	\node [below left] at (a13) {$(0, 2)$};\node [below left] at (a12) {$(0, 1)$};
	\node [above right] at (a1) {$(0, 3)$};\node [above right] at (a2) {$(1, 3)$};
	\node [below right] at (a3) {$(1, 2)$};
	
		\end{tikzpicture}
	\caption {Embedding  of $C_\Gamma(\Z_2 \times \Z_4)$  in the projective plane.}
	\label{f7}
\end{figure}

 {\bf Case 4}: $R \cong \Z_5$. Note that $\Gamma(R, \{-1\})$ is a subgraph of $C_\Gamma(R)$ and by Lemma \ref{t0.1},
 $\Gamma(R, \{-1\}) \cong K_5$. Thus, $C_\Gamma(R)\cong K_5$ and so by Lemma \ref{l1}(c), $C_\Gamma(R)$ is projective.
\end{proof}

\bibliography{refs}{}

\end{document}